\documentclass[twoside,11pt]{article}

	\usepackage[T1]{fontenc}	
	\usepackage[latin1]{inputenc}	
	\usepackage{vmargin}		
	\usepackage{fancyhdr}		
	\usepackage{amsfonts,amsthm,amsmath,stmaryrd, amssymb,mathrsfs} 
	\usepackage{graphicx, subfigure} 
	\usepackage{mfpic} 
	\usepackage{listings} 
	\setpapersize{custom}{21cm}{29.7cm}
	\setmarginsrb{30.5mm}{10mm}{30.5mm}{47mm}{10mm}{6mm}{0mm}{10mm}
	\usepackage[pdftex]{hyperref}
	\usepackage{color}
	
	\numberwithin{equation}{section}

\def\R{{\mathbb R}}

\def\N{{\mathbb N}}
\def\Z{{\mathbb Z}}
\def\d{\partial}
\def\a{\alpha}

\def\dsp{\displaystyle}
\def\be{\begin{equation}}
\def\ee{\end{equation}}

\def\op{{\rm op}}
\def\und{\underline}

\begin{document}

\theoremstyle{plain}
\newtheorem{theo}{Theorem}[section]

\theoremstyle{plain}
\newtheorem*{theo*}{Theorem}

\theoremstyle{plain} 
\newtheorem{defi}{Definition}[section]

\theoremstyle{plain}
\newtheorem{prop}[theo]{Proposition}

\theoremstyle{plain}
\newtheorem{coro}[theo]{Corollary}

\theoremstyle{plain}
\newtheorem{lemma}[theo]{Lemma}

\theoremstyle{definition}
\newtheorem*{condi}{Condition}

\theoremstyle{definition}
\newtheorem*{notation}{Notation}

\theoremstyle{plain} 
\newtheorem{remark}{Remark}

\theoremstyle{plain} 
\newtheorem{hypo}{Hypothesis}

\title{On the action of pseudo-differential operators on Gevrey spaces}

\author{Baptiste Morisse \thanks{School of Mathematics, Cardiff University - \href{mailto:morisseb@cardiff.ac.uk}{morisseb@cardiff.ac.uk}. The author is supported by the EPSRC grant "Quantitative Estimates in Spectral Theory and Their Complexity" (EP/N020154/1). 
The author thanks his PhD advisor Benjamin Texier for all the remarks on this work, and Nicolas Lerner for interesting discussions.}
}

\date{\today}

\maketitle

\begin{abstract}
	In this paper we study the action of pseudo-differential operators acting on Gevrey spaces. We introduce classes of classical symbols with spatial Gevrey regularity. As the spatial Gevrey regularity of a symbol $p(\cdot,\xi)$ may depend on the frequency $\xi$, the action of the associated pseudo-differential operator $\op(p)$ may induce a loss of regularity. The proof is based on a para-product decomposition.  
\end{abstract}


\setcounter{tocdepth}{2}
\tableofcontents



\section{Introduction}


This paper is devoted to the study of a class of pseudo-differential operators acting in Gevrey spaces. We assume that the operators have symbols which are Gevrey regular in the spatial variable $x\in\R^{d}$ and satisfy estimates in $(x,\xi)$ derivatives which are analogous to the ones enjoyed by symbols of the classical classes $S^{m}_{\rho,\delta}$. These symbols are precisely defined in Section \ref{3.section.class}. We give three results:
\begin{itemize}
	\item The first, Proposition \ref{3.prop.DmFtau}, states that if a function $F$ belongs to $H^{m}{\rm \mathcal{G}}^{\sigma}_{\tau}$ (a
Gevrey space with Sobolev correction, defined in Section \ref{3.section.class}), then the operator $e^{\tau D^{\sigma}} F e^{-\tau D^{\sigma}}$ acts continuously in $H^{m}(\R^{d})$. This result appeared in slightly different form in the article \cite{bedrossian2013landau} by Bedrossian, Masmoudi and Mouhot (see Lemma 3.3 therein). Its proof relies on a para-product decomposition and precise triangle-like inequalities
in the spirit of \cite{bedrossian2013landau}. 
	\item Our second and main result, Theorem \ref{3.theo.action.delta}, describes the action of operators with symbols in ${\rm S}^{0}_{\rho,\delta}{\rm G}^{s}_{R} $
(classical symbols with Gevrey regularity, defined in Section \ref{3.section.class}) on Gevrey spaces. The proof relies again on a para-product
decomposition.
	\item The third result is Lemma \ref{3.lemma.conjugation}. Here we give precise bounds for the symbol
of $e^{\tau D^{\sigma}} {\rm op}(p) e^{-\tau D^{\sigma}}$, where $p$ belongs to ${\rm S}^{0}_{\rho,0}{\rm G}^{s}_{R} $. This completes Lemma 7.1 of the article \cite{bronshtein} by Colombini, Nishitani and Rauch.

\end{itemize}

A classical reference on Gevrey spaces is Rodino's book \cite{rodino1993linear}. See also the paper \cite{hua2001paradifferential} by Hua and Rodino, where slightly less general classes of symbols are studied. Questions about the action of pseudo-differential operators in Gevrey spaces naturally arise from the study of the Gevrey well-posedness of the Cauchy problem for first-order systems. The aforementioned article \cite{bronshtein} focuses on Gevrey well-posedness, and so does our own line of research \cite{morisse2016I}, \cite{morisse2016II}, \cite{morisse2016lemma} and \cite{morisse2016IIz}.


\section{Classes of Gevrey regular symbols}
\label{3.section.class}



\subsection{Gevrey spaces}


We start by two definitions of Gevrey spaces, one on the spatial side, the other on the Fourier side.

\begin{defi}[Gevrey spaces: the spatial viewpoint]
	\label{3.defi.gevrey.x}
	For any $s\in[1,\infty)$, we define $G^{s}_{R}$ to be the space of smooth functions $f$ such that, for any compact set $K$ of $\R^{d}$, there are two positive constants $C_{K}$ and $R_{K}$ for which there holds
	\be
		\label{3.gevrey.x.ineq}
		|\d_{x}^{\a} f|_{L^{\infty}(K)} \leq C_{K} R_{K}^{|\a|} |\a|!^{s} \quad , \quad \forall \,\a\in\N^{d} .
	\ee 
	
	\noindent We call $s$ the Gevrey (regularity) index, and $R_{K}^{-1}$ the Gevrey radius.
	
	For $B$ a compact set of $\R^{d}$ and $R>0$ being fixed, we define $G^{s}_{R}(B)$ the space of smooth functions $f$ compactly supported on $B$ and being in $G^{s}$. That is, there is a positive constant $C>0$ for which there holds
	\be
		\label{3.gevrey.x.ineq.B}
		|\d_{x}^{\a} f|_{L^{\infty}(B)} \leq C R^{|\a|} |\a|!^{s} \quad , \quad \forall \,\a\in\N^{d} .
	\ee 
	
	\noindent The space $G^{s}_{R}(B)$ can be associated with the norm defined by 
	\be
		| f |_{s,R} = \sup_{\a \in \N^{d}} |\d_{x}^{\a} f|_{L^{\infty}(B)} \left( R^{|\a|} |\a|!^{s} \right)^{-1} .
	\ee
\end{defi}

We define
\be
	\langle \xi \rangle = (1 + |\xi|^2)^{1/2} \quad , \quad \forall\,\xi\in\R^{d} .
\ee

\begin{defi}[Gevrey spaces: the Fourier viewpoint]
	\label{3.defi.gevrey.xi}
	For any $\sigma \in(0,1]$ and $\tau >0$, we define $\mathcal{G}^{\sigma}_{\tau}$ to be the space of functions $f\in L^2$ such that $\exp( \tau \langle\cdot\rangle^{\sigma}) \hat{f}$ is in $L^2$. The associated norm is defined by
	\be
		\left| f \right|_{\sigma,\tau} = \left| e^{ \tau \langle \cdot \rangle^{\sigma}} \widehat{f} \,\right|_{L^2} .
	\ee
	
	We call $\tau$ the Gevrey radius and $\sigma$ the Gevrey (regularity) index.
\end{defi}

Both previous definitions of Gevrey functions are linked, as shown by the following classical result (see \cite{rodino1993linear}):
\begin{prop}[$G^{s}_{R}(B)$ is included in $\mathcal{G}^{1/s}_{\tau}$]
	\label{3.prop.x.in.xi}
	For any compact set $B$ of $\R^{d}$, the space $G^{s}_{R}(B)$ is included in the space $\mathcal{G}^{\sigma}_{\tau}$ for $\sigma = 1/s$ and $\tau < sR^{-1/s}$.
	Moreover there holds
	$$
		|u|_{\sigma,\tau} \leq |B|^{1/2} C\left(\tau s^{-1}R^{1/s}\right) |u|_{s,R} 
	$$
	
	\noindent with
	\be
		\label{3.def.C.gevrey}
		C\left( y \right) \lesssim \frac{1}{1 - y} P\left( \frac{y}{1 - y} \right)
	\ee
	
	\noindent where $P$ is a polynomial with degree at most $ \left\lceil  (3s-1)/2 \right\rceil $, and the implicit constant depends only on the Gevrey index $s$.
\end{prop}

\begin{proof}

	First we write $|u|_{\sigma,\tau} = e^{\tau} \left|e^{\tau ( \langle \xi \rangle^{\sigma} -1) } \widehat{u}(\xi) \right|_{L^2}$ and there holds
	$$
		|u|_{\sigma,\tau} \leq e^{\tau} \sum_{n \geq 0} \frac{\tau^{n}}{n!} \left| \left(\langle \xi \rangle^{\sigma} -1 \right)^n \widehat{u}(\xi) \right|_{L^2} .
	$$
	
	\noindent Next we compute $	\langle \xi \rangle^{\sigma} -1 = \sigma \int_{0}^{1} \langle t \xi \rangle^{\sigma -2} t |\xi|^2 dt$ which implies, as $\sigma - 2 <0$, that
	\begin{eqnarray*}
		\langle \xi \rangle^{\sigma} -1 & \leq &  \sigma \int_{0}^{1} \left( t |\xi| \right)^{\sigma -2} t |\xi|^2 dt \\
			& \leq & |\xi|^{\sigma} \int_{0}^{1} \sigma t^{\sigma-1} dt \\
			& \leq & |\xi|^{\sigma} .
	\end{eqnarray*}
	
	\noindent Thus there holds 
	$$
		|u|_{\sigma,\tau} \leq e^{\tau} \sum_{n \geq 0} \frac{\tau^{n}}{n!} \left| |\xi|^{n\sigma} \widehat{u}(\xi) \right|_{L^2}.
	$$
	
	\noindent Let $n$ be given in the following. By the support of $u$ and inequalities \eqref{3.gevrey.x.ineq}, there holds
	$$
		\left| |\xi|^{n\sigma} \widehat{u}(\xi) \right|_{L^2} \leq |B|^{1/2} \,|u|_{s,R} R^{m} m!^{s}
	$$
	
	\noindent for any $m \geq n\sigma$. Thus
	\begin{eqnarray*}
		\frac{\tau^{n}}{n!} \left| |\xi|^{n\sigma} \widehat{u}(\xi) \right|_{L^2} & \leq & \frac{m!^{s}}{n!} |B|^{1/2} \,|u|_{s,R} \tau^{n}R^{m} .
	\end{eqnarray*}
	
	\noindent Using Stirling's formula with $m < n\sigma +1$, there is $\delta >0$ such that
	\begin{eqnarray*}
		\frac{m!^{s}}{n!} & \leq & (1+\delta)^s \frac{1}{n!} \left( \frac{n\sigma+1}{e} \right)^{(n\sigma+1)s} \left( 2\pi(n\sigma+1) \right)^{s/2} \\
			& \leq & (1+\delta)^{2s} \left( \frac{n\sigma+1}{e} \right)^{(n\sigma+1)s} \left( \frac{n}{e} \right)^{-n} \left( 2\pi(n\sigma+1) \right)^{s/2} \left( 2\pi n \right)^{-1/2} .
	\end{eqnarray*}
	
	\noindent As $s\sigma = 1$, there holds
	$$
		\left( \frac{n\sigma+1}{e} \right)^{(n\sigma+1)s} \left( \frac{n}{e} \right)^{-n} \leq \sigma^{n} (n\sigma+1)^{s}e^{1/\sigma-s} .
	$$
	
	\noindent This implies finally
	$$
		\frac{m!^{s}}{n!} \lesssim \sigma^{n} n^{(3s-1)/2}
	$$
	
	\noindent hence
	$$
		\frac{\tau^{n}}{n!} \left| |\xi|^{n\sigma} \widehat{u}(\xi) \right|_{L^2} \lesssim |B|^{1/2} \,|u|_{s,R} \,n^{(3s-1)/2}\,\left(\sigma\tau R^{\sigma} \right)^{n} .
	$$
	\noindent It now suffices to sum in $n\in\N$.
\end{proof}

We recall here some useful inequalities when dealing with Gevrey spaces $\mathcal{G}^{\sigma}_{\tau}$. 
\begin{lemma}
	\label{3.lemma.tri.ineq} 
	${}^{}$
	\begin{enumerate}
		\item Let $\sigma \in (0,1)$, $\xi$ and $\eta$ in $\R^{d}$ such that $|\xi - \eta| \leq \frac{\dsp 1}{\dsp K} |\eta|$ for some $K >1$. Then
		\be
			\label{3.ineq.tri.1}
			\left| \langle \xi \rangle^{\sigma} - \langle \eta \rangle^{\sigma} \right| \leq \left( K^{\sigma} - (K-1)^{\sigma} \right) \langle \xi - \eta \rangle^{\sigma} .
		\ee
		
		\noindent Note that $K^{\sigma} - (K-1)^{\sigma} < 1$ for any $K>1$.
		
		\item Let $\sigma \in (0,1)$, $\xi$ and $\eta$ in $\R^{d}$ such that $\frac{\dsp 1}{\dsp K}|\xi - \eta| \leq |\eta| \leq K |\xi -\eta|$ for some $K >1$. Then
		\be
			\label{3.ineq.tri.2}
			\langle \xi \rangle^{\sigma} \leq \langle \eta \rangle^{\sigma} + c' \langle \xi - \eta \rangle^{\sigma}
		\ee
		
		\noindent for some $c' \in (0,1)$ depending on $K$.
		
		\item For any $\xi\in\R^{d}$, $\sigma\in(0,1)$, $\tau >0$ and $m \geq 0$, there holds
		\be
			\label{3.ineq.poly.gevrey}
			\langle \xi \rangle^{m} \lesssim \tau^{-m/\sigma} e^{\tau \langle \xi \rangle^{\sigma}}
		\ee
	
	\end{enumerate}
\end{lemma}

\begin{remark}
	Note that the first point in the previous Lemma does not hold when $\sigma = 1$, i.e. in the analytic regularity.
\end{remark}

\begin{proof}
	Denote $f(t) = \langle \eta + t(\xi-\eta) \rangle^{\sigma}$. The function $f$ is differentiable on $[0,1]$, and there holds
	$$
		f(1) - f(0) = \sigma (\xi- \eta)\cdot \int_{0}^{1} (\eta + t(\xi - \eta))\langle \eta + t(\xi - \eta) \rangle^{\sigma - 2} dt.
	$$
	
	\noindent As $|\xi - \eta| \leq \frac{\dsp 1}{\dsp K} |\eta|$, there holds $ |\eta + t(\xi - \eta)| \geq (K-t) |\xi - \eta|$ hence
	\begin{eqnarray*}
		|f(1) - f(0)| & \leq & \sigma |\xi- \eta|\int_{0}^{1} \langle \eta + t(\xi - \eta) \rangle^{\sigma - 1} dt \\
			& \leq & \sigma \langle \xi- \eta \rangle^{\sigma} \int_{0}^{1} (K-t)^{\sigma-1} dt
	\end{eqnarray*}
	
	\noindent which is \eqref{3.ineq.tri.1}. We now turn to the proof of \eqref{3.ineq.tri.2}, from $|\eta| \geq K^{-1}|\xi-\eta|$ we deduce
	$$
		1 + |\eta|^2 \geq 1 + K^{-2}|\xi - \eta|^2 = K^{-2}\left( K^2 + |\xi - \eta|^2 \right) \geq K^{-2} \langle \xi - \eta \rangle^{2}
	$$
	
	\noindent where we used $K>1$. Thus, since $0 < \sigma$, 
	$$
		\langle \eta \rangle^{\sigma} \geq K^{-\sigma} \langle \xi - \eta \rangle^{\sigma}.
	$$
	
	\noindent This implies 
	$$
		\langle \eta \rangle^{\sigma} + \langle \xi - \eta \rangle^{\sigma} \geq \left(1 + K^{-\sigma} \right) \langle \xi - \eta \rangle^{\sigma}.
	$$
	
	\noindent Now assume in addition 
	\be
		\label{3.local.ineq}
		|\xi| \leq c |\xi - \eta| , \quad \text{for some } c>0.
	\ee
	
	\noindent If \eqref{3.local.ineq} holds with some $c \geq 1$, then it holds a fortiori with $c>1$. Thus we may assume \eqref{3.local.ineq} for some $c>1$, and then
	$$
		\langle \xi \rangle \leq c \langle \xi - \eta \rangle,
	$$
	
	\noindent so that 
	$$
		\langle \xi \rangle^{\sigma} \leq c^{\sigma}\langle \xi - \eta \rangle^{\sigma} \leq c^{\sigma}\left( 1 + K^{-\sigma} \right)^{-1}\left( \langle \eta \rangle^{\sigma} + \langle \xi - \eta \rangle^{\sigma} \right) .
	$$
	
	\noindent Thus we are done if \eqref{3.local.ineq} holds with
	$$
		1 < c < \left( 1 + K^{-\sigma} \right)^{1/\sigma} .
	$$
	
	\noindent Otherwise, there holds
	$$
		|\xi| \geq \left( 1 + K^{-\sigma} \right)^{1/\sigma} |\xi - \eta| =: \widetilde{c} \,|\xi-\eta|,
	$$
	
	\noindent and since $\widetilde{c} > 1$, we may then apply \eqref{3.ineq.tri.1}. This yields
	$$
		\langle \xi \rangle^{\sigma} \leq \langle \eta \rangle^{\sigma} + \left( \widetilde{c}^{\,\sigma} - (\widetilde{c} - 1)^{\sigma} \right) \langle \xi - \eta \rangle^{\sigma} 
	$$
	
	\noindent and the result follows. The proof of \eqref{3.ineq.poly.gevrey} is trivial, hence omitted.
	
\end{proof}

\begin{remark}
	Inequality \eqref{3.ineq.tri.1} is somehow similar to inequality (3.11) in {\rm \cite{bedrossian2013landau}}, which we reproduce here:
	$$
		\left| \langle \xi \rangle^{\sigma} - \langle \eta \rangle^{\sigma} \right| \leq \frac{\sigma}{(K-1)^{1-\sigma}} \langle \xi - \eta \rangle^{\sigma}
	$$
	
	\noindent Note that the coefficient $\frac{\sigma}{(K-1)^{1-\sigma}}$ may be strictly greater than $1$. Inequality \eqref{3.ineq.tri.2} is similar to inequality (3.12) in {\rm \cite{bedrossian2013landau}}, which we reproduce here:
	$$
		\langle \xi \rangle^{\sigma} \leq \left( \frac{\langle \eta \rangle^{\sigma}}{\langle \xi \rangle^{\sigma}} \right)^{1-\sigma} \left( \langle \eta \rangle^{\sigma} + \langle \xi - \eta \rangle^{\sigma} \right)
	$$
	 
	\noindent for $|\eta| \geq |\xi - \eta|$. Again, the coefficient $\left( \frac{\langle \eta \rangle^{\sigma}}{\langle \xi \rangle^{\sigma}} \right)^{1-\sigma}$ may be strictly greater than 1. 
\end{remark}


\subsection{Classes of symbols}


We define a class of symbols $a(x,\xi)$ with Gevrey regularity in the spatial variable $x$.

\begin{defi}[Class of symbols with Gevrey regularity]
	\label{3.defi.class}
	For $s\in(1,\infty)$ and $R >0$, for $m\in\R$, $\rho$ and $\delta$ such that $0 < \delta < \rho \leq 1$, we define ${\rm S}^{m}_{\rho,\delta}{\rm G}^{s}_{R}$ to be the class of symbols $a(x,\xi)$ for which there is a bounded sequence of positive numbers $C_{\a,\beta}$ such that
	\be
		\label{3.class.inequality}
		|\d_{x}^{\a} \d_{\xi}^{\beta} a(x,\xi)| \leq C_{\a,\beta}R^{|\a + \beta|} |\a|!^{s} |\beta|! \langle \xi \rangle^{m-\rho|\beta| + \delta|\a|}
	\ee
	
	\noindent uniformly in $x\in\R^{d}$ and $\xi\in\R^{d}$. We denote
	\be
		\label{3.def.semi.norms}
		|a|_{\a,\beta} = \sup_{(x,\xi)\in \R^{d}\times\R^{d}} \left| R^{-|\a + \beta|} |\a|!^{-s} |\beta|!^{-1} \langle \xi \rangle^{-m+\rho|\beta| - \delta|\a|}\d_{x}^{\a} \d_{\xi}^{\beta} a(x,\xi) \right| .
	\ee
\end{defi}

\begin{remark}
	Note that the space $G^{s}_{R}$ is naturally in ${\rm S}^{0}_{0,0}{\rm G}^{s}_{R}$, with 
	$$
		|a|_{\a,\beta} \leq |a|_{s,R} \quad , \quad \forall \,(\a,\beta) \in \Z^{d} \times \Z^{d}.
	$$
	
	\noindent Moreover spaces ${\rm S}^{m}_{\rho,\delta}{\rm G}^{s}_{R}$ are naturally embedded in $S^{m}_{\rho,\delta}$.
\end{remark}

\begin{remark}
	\label{3.remark.2}
	A way to look at inequalities \eqref{3.class.inequality} is to put together the Gevrey term $R^{|\a|}|\a|!^{s}$ and the typical pseudo-differential term $\langle \xi \rangle^{\delta |\a|}$, which means that $a(\cdot,\xi)$ is in $G^{s}_{R \langle \xi \rangle^{\delta} }$ for all $\xi\in\R^{d}$: the Gevrey radius in $x$ of the symbol decreases with $|\xi|$ if $\delta >0$. 
\end{remark}

%
%
%
%
%
%
%
%


\section{Conjugation of a Gevrey function}
\label{3.section.conjugation}


We consider the Gevrey conjugation operator of a function $F$ in ${\rm \mathcal{G}}^{\sigma}_{\tau}$ with $\tau \geq 0$, and we denote
\be
	F^{(\tau)} = e^{ \tau D^{\sigma}} F \, e^{ -\tau D^{\sigma}}
\ee

\noindent where $D = {\rm op}( \langle \cdot \rangle)$.

\begin{prop}
	\label{3.prop.DmFtau}
	Assume that $D^{m} F \in {\rm \mathcal{G}}^{\sigma}_{\tau}$ for some $m\geq0$. Then, for any $v \in H^m$, there holds
	$$
		\left| F^{(\tau)} v \right|_{H^{m}} \lesssim \left| D^{m} F \right|_{\sigma,\tau} \left| v \right|_{L^2} + \left| F \right|_{\sigma,\tau} \left| v \right|_{H^m}  .
	$$
\end{prop}

\begin{remark}
	This implies in particular that $H^{m}{\rm \mathcal{G}}^{\sigma}_{\tau}$, the space of Gevrey functions with Sobolev correction of order $m$, is an algebra for any $m \geq 0$. With $m=0$ in proposition {\rm \ref{3.prop.DmFtau}}, we see that if $F$ is in ${\rm \mathcal{G}}^{\sigma}_{\tau}$, then $F^{(\tau)}$ operates in $L^2$.
\end{remark}


\begin{proof}
	In Fourier there holds
	\be
		\label{3.local.dmtau.f}
		\mathcal{F}\left( D^{m} F^{(\tau)} v \right)(\xi) = \int_{\eta} e^{ \tau \langle \xi \rangle^{\sigma} - \tau \langle \eta \rangle^{\sigma} } \langle \xi \rangle^{m} \widehat{F}(\xi - \eta) \widehat{v}(\eta) d\eta .
	\ee
	
	\noindent We use here a paraproduct decomposition and Lemma \ref{3.lemma.tri.ineq}. Let $K>1$. We divide the integral in $\eta \in \R^{d}$ in three frequency regions, defined as $\mathcal{R}_1 = \left\{ \eta : |\xi - \eta| \leq \frac{1}{K}|\eta| \right\}$, $\mathcal{R}_2 = \left\{ \eta : |\eta| \leq \frac{1}{K} |\xi - \eta| \right\}$ and $\mathcal{R}_3 = \left\{ \eta : \frac{1}{K} |\xi - \eta| < |\eta| < K |\xi - \eta| \right\}$. We consider then each region successively:
	
	$\bullet$ The case where $|\xi - \eta| \leq \frac{1}{K}|\eta|$: thanks to inequality \eqref{3.ineq.tri.1} in Lemma \ref{3.lemma.tri.ineq}, there is $c \in (0,1)$ such that $ \langle \xi \rangle^{\sigma} - \langle \eta \rangle^{\sigma} \leq c \langle \xi - \eta \rangle^{\sigma} $, hence
		$$
			e^{ \tau \langle \xi \rangle^{\sigma} - \tau \langle \eta \rangle^{\sigma} } \leq e^{ c\tau \langle \xi - \eta \rangle^{\sigma} }.
		$$
		
		\noindent Besides, in the region under consideration, the Sobolev term satisfies $\langle \xi \rangle^{m} \lesssim \langle \eta \rangle^{m} $. This implies that
		\begin{align*}
			& \int_{\eta \in \mathcal{R}_1} e^{ \tau \langle \xi \rangle^{\sigma} - \tau \langle \eta \rangle^{\sigma} } \langle \xi \rangle^{m} \left|\widehat{F}(\xi - \eta)\right| \left|\widehat{v}(\eta)\right| d\eta \\
			& \lesssim \int_{\eta \in \mathcal{R}_1} e^{ c\tau \langle \xi - \eta \rangle^{\sigma} } \langle \eta \rangle^{m}  \left|\widehat{F}(\xi - \eta)\right| \left|\widehat{v}(\eta)\right| d\eta \\
			& \lesssim \int_{\eta \in \mathcal{R}_1} e^{ -\tau(1-c) \langle \xi - \eta \rangle^{\sigma} } \, e^{ \tau \langle \xi - \eta \rangle^{\sigma} }\left|\widehat{F}(\xi - \eta)\right| \, \langle \eta \rangle^{m}\left|\widehat{v}(\eta)\right| d\eta .
		\end{align*}
		
		\noindent We use next Young's inequality to obtain
		\begin{align*}
			& \left| \int_{\eta \in \mathcal{R}_1} e^{ -\tau(1-c) \langle \xi - \eta \rangle^{\sigma} } \, e^{ \tau \langle \xi - \eta \rangle^{\sigma} }\left|\widehat{F}(\xi - \eta)\right| \, \langle \eta \rangle^{m}\left|\widehat{v}(\eta)\right| d\eta \right|_{L^2_{\xi}} \\
			& \lesssim \left| e^{ -\tau(1-c) \langle \cdot \rangle^{\sigma} } \, e^{ \tau \langle \cdot \rangle^{\sigma} }\left|\widehat{F}(\cdot)\right| \right|_{L^1} \left| v \right|_{H^m} \\
			& \lesssim \left| e^{ -\tau(1-c) \langle \cdot \rangle^{\sigma} } \right|_{L^2} \left| F \right|_{\sigma,\tau} \left| v \right|_{H^m}
		\end{align*}
		
		\noindent using the extra Gevrey weight $ e^{ -\tau(1-c) \langle \xi - \eta \rangle^{\sigma} } $.
		
	$\bullet$ The case where $|\eta| \leq \frac{1}{K} |\xi - \eta|$: thanks to inequality \eqref{3.ineq.tri.1} in Lemma \ref{3.lemma.tri.ineq}, there is $c \in(0,1)$ such that $ \langle \xi \rangle^{\sigma} - \langle \xi - \eta \rangle^{\sigma} \leq c \langle \eta \rangle^{\sigma} $, hence
	$$
		e^{\langle \xi \rangle^{\sigma} - \langle \xi - \eta \rangle^{\sigma}} \leq e^{c \langle \eta \rangle^{\sigma}}.
	$$
	
	\noindent Besides, in the region under consideration, the Sobolev term satisfies $\langle \xi \rangle^{m} \lesssim \langle \xi - \eta \rangle^{m} $. This implies that
		\begin{align*}
			& \int_{\eta \in \mathcal{R}_2} e^{ \tau \langle \xi \rangle^{\sigma} - \tau \langle \eta \rangle^{\sigma} } \langle \xi \rangle^{m}  \left|\widehat{F}(\xi - \eta)\right| \left|\widehat{v}(\eta)\right| d\eta \\
			& \lesssim \int_{\eta \in \mathcal{R}_2} e^{\tau \langle \xi \rangle^{\sigma} - \tau \langle \eta \rangle^{\sigma} - \tau \langle \xi - \eta \rangle^{\sigma} } \langle \xi - \eta \rangle^{m}  e^{\tau \langle \xi - \eta \rangle^{\sigma}} \left|\widehat{F}(\xi - \eta)\right| \left|\widehat{v}(\eta)\right| d\eta \\
			& \lesssim \int_{\eta \in \mathcal{R}_2} e^{- \tau (1-c) \langle \eta \rangle^{\sigma} } \langle \xi -\eta \rangle^{m} \, e^{\tau \langle \xi - \eta \rangle^{\sigma}} \left|\widehat{F}(\xi - \eta)\right| \, \left|\widehat{v}(\eta)\right| d\eta .
		\end{align*}
		
		\noindent We use next Young's inequality to obtain
		\begin{align*}
			& \left| \int_{\eta \in \mathcal{R}_2} e^{- \tau (1-c) \langle \eta \rangle^{\sigma} } \langle \xi -\eta \rangle^{m} \, e^{\tau \langle \xi - \eta \rangle^{\sigma}} \left|\widehat{F}(\xi - \eta)\right| \, \left|\widehat{v}(\eta)\right| d\eta \right|_{L^2_{\xi}} \\
			& \lesssim \left| e^{ -\tau(1-c) \langle \cdot \rangle^{\sigma} } \, \left|\widehat{v}(\eta)\right| \right|_{L^1} \left| D^{m} F \right|_{\sigma,\tau} \\
			& \lesssim \left| e^{ -\tau(1-c) \langle \cdot \rangle^{\sigma} } \right|_{L^2} \left| D^{m} F \right|_{\sigma,\tau} \left| v \right|_{L^2}
		\end{align*}
		
		\noindent using the extra Gevrey weight $ e^{ -\tau(1-c) \langle \xi - \eta \rangle^{\sigma} } $.
		
	$\bullet$ The case where $\frac{1}{K} |\xi - \eta| < |\eta| < K |\xi - \eta|$: thanks to inequality \eqref{3.ineq.tri.2} in Lemma \ref{3.lemma.tri.ineq}, there is $c' \in(0,1)$ such that $ \langle \xi \rangle^{\sigma} \leq c' \langle \xi - \eta \rangle^{\sigma} + \langle \eta \rangle^{\sigma} $, hence
	$$
		e^{ \tau \langle \xi \rangle^{\sigma} - \tau \langle \eta \rangle^{\sigma} } \leq e^{c' \langle \xi - \eta \rangle^{\sigma}} .
	$$
	
	\noindent Besides, in the region under consideration, the Sobolev term satisfies $\langle \xi \rangle^{m} \lesssim \langle \eta \rangle^{m} + \langle \xi - \eta \rangle^{m} $ where the implicit constant depends on $m$, thus
		\begin{align*}
			& \int_{\eta \in \mathcal{R}_3} e^{ \tau \langle \xi \rangle^{\sigma} - \tau \langle \eta \rangle^{\sigma} } \langle \xi \rangle^{m}  \left|\widehat{F}(\xi - \eta)\right| \left|\widehat{v}(\eta)\right| d\eta \\
			& \lesssim \int_{\eta \in \mathcal{R}_3} e^{ - (1-c')\tau \langle \xi - \eta \rangle^{\sigma} } \left( \langle \eta \rangle^{m} + \langle \xi - \eta \rangle^{m} \right)  e^{\tau \langle \xi - \eta \rangle^{\sigma}} \left|\widehat{F}(\xi - \eta)\right| \left|\widehat{v}(\eta)\right| d\eta \\
			& \lesssim \int_{\eta \in \mathcal{R}_3} e^{ - (1-c')\tau \langle \xi - \eta \rangle^{\sigma} }   e^{\tau \langle \xi - \eta \rangle^{\sigma}} \left|\widehat{F}(\xi - \eta)\right| \, \langle \eta \rangle^{m} \left|\widehat{v}(\eta)\right| d\eta \\
			& \quad + \int_{\eta \in \mathcal{R}_3} e^{ - (1-c')\tau \langle \xi - \eta \rangle^{\sigma} } \langle \xi - \eta \rangle^{m}  e^{\tau \langle \xi - \eta \rangle^{\sigma}} \left|\widehat{F}(\xi - \eta)\right| \left|\widehat{v}(\eta)\right| d\eta
		\end{align*}
		
		\noindent We use next Young's inequality to obtain
		\begin{align*}
			& \left| \int_{\eta \in \mathcal{R}_3} e^{ \tau \langle \xi \rangle^{\sigma} - \tau \langle \eta \rangle^{\sigma} } \langle \xi \rangle^{m}  \left|\widehat{F}(\xi - \eta)\right| \left|\widehat{v}(\eta)\right| d\eta \right|_{L^2_{\xi}} \\
			& \lesssim \left| e^{ - (1-c')\tau \langle \cdot \rangle^{\sigma} } \right|_{L^2} \left( \left| D^{m} F \right|_{\sigma,\tau} \left| v \right|_{L^2} + \left| F \right|_{\sigma,\tau} \left| v \right|_{H^m} \right) 
		\end{align*}
		
		The result follows from \eqref{3.local.dmtau.f}, viewed as an integral over $\mathcal{R}_1 \cup\mathcal{R}_2 \cup\mathcal{R}_3$.
	
\end{proof}


\section{Action of pseudo-differential operators on Gevrey spaces}
\label{3.section.action}


In this Section, we consider symbols in ${\rm S}^{0}_{\rho,\delta}{\rm G}^{s}_{R}$ with compact support $B$ of $\R^{d}_{x}$, uniformly in $\xi\in\R^d$. This additional assumption on the support of the symbol allows to use Proposition \ref{3.prop.x.in.xi}, parlaying the spatial Gevrey regularity into a Fourier Gevrey regularity for $a(\cdot,\xi)$. We may then use an adapted paraproduct decomposition to prove the continuous action of operators with symbols in ${\rm S}^{0}_{\rho,\delta}{\rm G}^{s}_{R}$. 

In all the following we consider quantizations of the type
$$
	\op_{h}(a)u(x) = \int e^{i (x-y)\cdot\eta}a( x - h(x-y) , \eta) u(y) dy d\eta .
$$

\noindent with $h\in[0,1]$. First we prove this result in the particular case $\rho = 1$, $\delta = 0$. 

\begin{theo}[Action of $S^{0}_{1,0}G^{s}_{R}$ on $\mathcal{G}_{\tau}^{\sigma}$]
	\label{3.theo.action.0}
	Let $s\in(1,\infty)$ and $R>0$. Let $a$ be in ${\rm S}^{0}_{1,0}{\rm G}^{s}_{R}$, constant outside a compact set $B$ of $\R^{d}_{x}$, uniformly in $\xi\in\R^d$. Then for any $
		\tau < sR^{-1/s}$ and $\sigma = \frac{\dsp 1}{\dsp s} $, the operator ${\rm op}_{h}(a)$ acts continuously on ${\rm \mathcal{G}}_{\tau}^{\sigma}$ with norm
	$$
		\| {\rm op}_{h}(a) \|_{\mathcal{L}({\rm \mathcal{G}}_{\tau}^{\sigma})} \lesssim_{h} |B|^{1/2} C\left(\tau s^{-1}R^{1/s}\right) \sup_{\a\in\N^{d}} |a|_{\a,0}
	$$
	
	\noindent where $C\left(\tau s^{-1}R^{1/s}\right)$ is defined in \eqref{3.def.C.gevrey}.
\end{theo}

\begin{proof}
	First, for fixed $\xi\in\R^d$, as $a(\cdot,\xi)$ is in ${\rm G}^{s}_{R}$ with compact support, Proposition \ref{3.prop.x.in.xi} implies that $\widehat{a}(\cdot,\xi)$, the Fourier transform with respect to $x$ of $a(\cdot,\xi)$, is in ${\rm \mathcal{G}}^{\sigma}_{\tau}$ uniformly in $\xi\in\R^{d}$, with $\sigma = 1/s$ and $\tau < sR^{-1/s}$. That is, we may write
	\be
		\label{3.local.def.F}
		\widehat{a}(\zeta,\xi) = F_{\xi}(\zeta) ,
	\ee
	
	\noindent where for fixed $\xi\in\R^d$, $F_{\xi}(\cdot)$ belongs to ${\rm \mathcal{G}}^{\sigma}_{\tau}$ with the uniform (in $\xi$) bound
	$$
		\left| F_{\xi}(\cdot) \right|_{\sigma, \tau} \leq |B|^{1/2} C\left(\tau s^{-1}R^{1/s}\right) |a(\cdot,\xi)|_{s,R} .
	$$
	
	\noindent thanks to Proposition \ref{3.prop.x.in.xi}. By definition \eqref{3.def.semi.norms} of the semi-norms in Definition \ref{3.defi.class}, there holds
	\be
		\label{3.local.bound.gevrey}
		\left| F_{\xi}(\cdot) \right|_{\sigma, \tau} \leq |B|^{1/2} C\left(\tau s^{-1}R^{1/s}\right) \sup_{\a\in\N^{d}} |a|_{\a,0} .
	\ee
	
	Let $u$ be in $G^{\sigma}_{\tau}$, and denote $v(\eta) = e^{\tau \langle \eta \rangle^{\sigma}} \hat{u}(\eta)$ which is in $L^2$. We compute the Fourier transform of ${\rm op}_{h}(p)u$. The case $h=0$ is simple, as the Fourier transform of ${\rm op}_{0}(p)u$ is
	$$
		\int_{\eta} \widehat{a}(\xi-\eta, \xi) \hat{u}(\eta) d\eta.
	$$
	
	\noindent Thus there holds
	$$
		e^{\tau \langle \xi \rangle^{\sigma}} \mathcal{F}\left( {\rm op}_{0}(a)u \right)(\xi) = \int_{\eta} e^{ \tau \langle \xi \rangle^{\sigma} - \tau \langle \eta \rangle^{\sigma}} F_{\xi}(\xi - \eta) v(\eta) d\eta .
	$$
	
	\noindent Proposition \ref{3.prop.DmFtau} now yields the result for $h=0$, since the bound \eqref{3.local.bound.gevrey} is uniform in $\xi$.

	For $h \in(0,1]$, the computation of the Fourier transform is more delicate, and there holds 
	\begin{eqnarray*}
		\mathcal{F}\left( \op_{h}(a)u \right)(\xi) & = & \int e^{-ix\cdot\xi} e^{i (x-y)\cdot\eta} a( x - h(x-y) , \eta) u(y) \, dy d\eta dx \\
			& = & \int e^{-ix\cdot\xi} e^{i (x-y)\cdot\eta}e^{iy\cdot\zeta} a( x - h(x-y) , \eta) \widehat{u}(\zeta) \, d\zeta dy d\eta dx .
	\end{eqnarray*}
	
	\noindent We define new variables, putting
	$$
		\widetilde{x} = (1-h)x + hy \quad , \quad \widetilde{y} = (1-h)x - hy
	$$
	
	\noindent which leads to
	\begin{align*}
		& \int e^{-ix\cdot\xi} e^{i (x-y)\cdot\eta}e^{iy\cdot\zeta} a( x - h(x-y) , \eta) \widehat{u}(\zeta) \, d\zeta dy d\eta dx \\
		& = \int e^{i\widetilde{y}\cdot \varphi_1(\xi,\zeta,\eta) } \, e^{i\widetilde{x} \cdot \varphi_2(\xi,\zeta,\eta) } \, a( \widetilde{x} , \eta) \widehat{u}(\zeta) \, d\zeta d\widetilde{y} d\eta d\widetilde{x} 
	\end{align*}
	
	\noindent with 
	\be
		\label{def.phi.1}
		\varphi_1(\xi,\zeta,\eta) = -\frac{1}{2(1-h)}\xi - \frac{1}{2h}\zeta + \frac{1}{2h(1-h)}\eta 
	\ee
	
	\noindent and 
	\be
		\label{def.phi.2}
		\varphi_2(\xi,\zeta,\eta) = -\frac{1}{2(1-h)}\xi + \frac{1}{2h}\zeta + \frac{2h-1}{2h(1-h)}\eta .
	\ee
	
	\noindent As the integrand depends on $\tilde{y}$ only through the phase term, there holds
	\begin{align*}
		& \int e^{i\widetilde{y}\cdot \varphi_1(\xi,\zeta,\eta) } \, e^{i\widetilde{x} \cdot \varphi_2(\xi,\zeta,\eta) } \, a( \widetilde{x} , \eta) \widehat{u}(\zeta) \, d\zeta d\widetilde{y} d\eta d\widetilde{x} \\
		& = \int_{ \{ (\zeta,\eta) \,:\, \varphi_1(\xi,\zeta,\eta) = 0 \} } \, e^{i\widetilde{x} \cdot \varphi_2(\xi,\zeta,\eta) } \, a( \widetilde{x} , \eta) \widehat{u}(\zeta) \, d\zeta d\eta d\widetilde{x} \\
		& = \int_{ \{ (\zeta,\eta) \,:\, \varphi_1(\xi,\zeta,\eta) = 0 \} } \, e^{i\widetilde{x}\cdot\left( \zeta - \eta \right)/h } \, a( \widetilde{x} , \eta) \widehat{u}(\zeta) \, d\zeta d\eta d\widetilde{x}
	\end{align*}
	
	\noindent by definition \eqref{def.phi.2} of $\varphi_2$. Hence finally
	$$
		\mathcal{F}\left( \op_{h}(a)u \right)(\xi) = \int_{ \{ (\zeta,\eta) \,:\, \varphi_1(\xi,\zeta,\eta) = 0 \} } \, \widehat{a} \left( (\zeta-\eta)/h , \eta \right) \widehat{u}(\zeta) \, d\zeta d\eta .
	$$
	
	As in the proof for $h=0$, we write 
	\begin{align*}
		& e^{\tau \langle \xi \rangle^{\sigma}} \mathcal{F}\left( \op_{h}(a)u \right)(\xi) \\
		& = \int_{ \{ (\zeta,\eta) \,:\, \varphi_1(\xi,\zeta,\eta) = 0 \} } \, e^{ \tau \langle \xi \rangle^{\sigma} - \tau \langle \zeta \rangle^{\sigma} } \, \widehat{a} \left( (\zeta-\eta)/h , \eta \right) v(\zeta) \, d\zeta d\eta .
	\end{align*}
	
	\noindent On the surface $ \{ (\zeta,\eta) \,:\, \varphi_1(\xi,\zeta,\eta) = 0 \}$ there holds
	$$
		\xi = - \frac{1-h}{h}\zeta + \frac{1}{h}\eta = \frac{1}{h}\left( \eta - \zeta \right) + \zeta		
	$$
	
	\noindent hence
	$$
		e^{\tau \langle \xi \rangle^{\sigma}} \mathcal{F}\left( \op_{h}(a)u \right)(\xi) = \int_{ \{ (\zeta,\eta) \,:\, \varphi_1(\xi,\zeta,\eta) = 0 \} } \, e^{ \tau \langle \xi \rangle^{\sigma} - \tau \langle \zeta \rangle^{\sigma} } \, F_{\eta}(\xi - \zeta) v(\zeta) \, d\zeta d\eta 
	$$
	
	\noindent with $F$ defined in \eqref{3.local.def.F}. We may then conclude in the same as the case $h=0$.
	
\end{proof}

In the general case $0 < \delta < \rho \leq 1$, Remark \ref{3.remark.2} indicates a potential obstruction for the Gevrey index. This is made precise in the following

\begin{theo}[Action of $S^{0}_{\rho,\delta}G^{s}_{R}$ on $\mathcal{G}_{\tau}^{\sigma}$]
	\label{3.theo.action.delta}
	Let $s\in(1,\infty)$, $R>0$ and $0 < \delta < \rho \leq 1$. Let $a$ be in ${\rm S}^{0}_{\rho,\delta}{\rm G}^{s}_{R}$, constant outside a compact set $B$ of $\R^{d}_{x}$, uniformly in $\xi\in\R^d$. Then for any 
	$$
		\sigma \leq (1-\delta)/s \quad \text{and} \quad \tau' < \tau  < s R^{-1/s} 
	$$
	
	\noindent the operator ${\rm op}(a)$ acts continuously from ${\rm \mathcal{G}}_{\tau}^{\sigma}$ into ${\rm \mathcal{G}}_{\tau'}^{\sigma}$ with norm
	$$
		||{\rm op}_{0}(a)||_{\mathcal{L}({\rm \mathcal{G}}_{\tau}^{\sigma} , {\rm \mathcal{G}}_{\tau'}^{\sigma})} \lesssim |B|^{1/2} C\left(\tau s^{-1}R^{1/s}\right) \sup_{\a\in\N^{d}} |a|_{\a,0} 
	$$
	
	\noindent where $C\left(\tau s^{-1}R^{1/s}\right)$ is defined in \eqref{3.def.C.gevrey}.

\end{theo}

\begin{proof}

	First, for fixed $\xi\in\R^d$, as $a(\cdot,\xi)$ is in ${\rm G}^{s}_{R}$ with compact support, Proposition \ref{3.prop.x.in.xi} and Remark \ref{3.remark.2} implies that $\widehat{a}(\cdot,\xi)$, the Fourier transform with respect to $x$ of $a(\cdot,\xi)$, is in ${\rm \mathcal{G}}^{\sigma}_{\tau \langle \xi \rangle^{-\delta/s}}$ uniformly in $\xi\in\R^{d}$, with $\sigma = 1/s$ and $\tau < sR^{-1/s}$. That is, we may write
	$$
		e^{ \tau \langle \xi \rangle^{-\delta/s} \langle \zeta \rangle^{1/s} } \hat{a}(\zeta, \xi) = F_{\xi}(\zeta) ,
	$$
	
	\noindent where for fixed $\xi\in\R^d$, $F_{\xi}(\cdot)$ belongs to $L^2$ with the uniform (in $\xi$) bound
	$$
		\left| F_{\xi}(\cdot) \right|_{L^2} \leq |B|^{1/2} C\left(\tau s^{-1}R^{1/s}\right) |a(\cdot,\xi)|_{s,R} .
	$$
	
	\noindent thanks to Proposition \ref{3.prop.x.in.xi}. By definitions \eqref{3.def.semi.norms} of the semi-norms in Definition \ref{3.defi.class}, there holds
	\be
		\label{3.local.bound.gevrey.2}
		\left| F_{\xi}(\cdot) \right|_{\sigma, \tau} \leq |B|^{1/2} C\left(\tau s^{-1}R^{1/s}\right) \sup_{\a\in\N^{d}} |a|_{\a,0} .
	\ee

	Let $u$ be in ${\rm \mathcal{G}}^{\sigma}_{\tau}$. Denoting $v(\eta) = e^{\tau \langle \eta \rangle^{\sigma}} \hat{u}(\eta)$, there holds
	$$
		e^{\tau' \langle \xi \rangle^{\sigma}} \mathcal{F}\left( {\rm op}_{0}(a)u \right)(\xi) = \int_{\eta} e^{ \tau' \langle \xi \rangle^{\sigma} - \tau \langle \eta \rangle^{\sigma} - \tau \langle \xi \rangle^{-\delta/s} \langle \xi - \eta \rangle^{1/s} } F_{\xi}(\xi - \eta) v(\eta) d\eta
	$$
		
		We now decompose the integral into three regions, as in the proof of Proposition \ref{3.prop.DmFtau}. Once we derive appropriate bounds on the exponential factor
	$$
		W(\tau' ; \tau) := \exp \left( \tau' \langle \xi \rangle^{\sigma} - \tau \langle \eta \rangle^{\sigma} - \tau \langle \xi \rangle^{-\delta/s} \langle \xi - \eta \rangle^{1/s} \right)
	$$
	\noindent the result follows from \eqref{3.local.bound.gevrey.2} by application of Young's inequality, as in the proof of Proposition \ref{3.prop.DmFtau}. Thus we focus only on the above exponential factor. Here the multiplicative coefficient $K > 1$ is chosen in terms of $\tau$ and
$\tau'$ .
		
	\medskip
	
	$\bullet$ The case where $|\xi - \eta| \leq \frac{1}{K}|\eta|$: here the weight $\langle \xi \rangle^{-\delta/s}$ in the Gevrey radius of $\hat{a}(\cdot,\xi)$ is small, and 
	$$
		W(\tau';\tau) \leq \exp \left( \tau' \langle \xi \rangle^{\sigma} - \tau \langle \eta \rangle^{\sigma} \right).
	$$
	
	\noindent With \eqref{3.ineq.tri.1} there holds
	$$
		W(\tau';\tau) \leq \exp \left( \tau'\left( K^{\sigma} - (K-1)^{\sigma} \right) \langle \xi - \eta \rangle^{\sigma} -(\tau - \tau')\langle \eta \rangle^{\sigma} \right)
	$$
	
	\noindent Using now $|\xi - \eta| \leq \frac{1}{K}|\eta|$, there holds $\langle \xi - \eta \rangle^{\sigma} \leq \langle \eta \rangle^{\sigma}$, so that
	$$
		W(\tau';\tau) \leq \exp \left( -\left( \tau - (1 + \left( K^{\sigma} - (K-1)^{\sigma} \right))\tau' \right)\langle \eta \rangle^{\sigma} \right).
	$$
	
	\noindent For $K$ large enough, depending only on $\tau'$ and $\tau$, with $\tau' < \tau$, there holds $ \tau - (1 + \left( K^{\sigma} - (K-1)^{\sigma} \right))\tau' $, thus $|W|_{L^2_{\eta}} < \infty$.
			
		$\bullet$ The case where $|\eta| \leq \frac{1}{K} |\xi - \eta|$: since $|\eta| \leq K^{-1}|\xi - \eta|$, there holds $|\xi| \leq (1+ K^{-1})|\xi - \eta|$, hence $\langle \xi \rangle \leq (1 + K^{-1})\langle \xi - \eta \rangle$. Thus, with $\sigma \leq (1-\delta)/s$, we find the bound
		$$
			W(\tau' ; \tau) \leq \exp \left( \tau' \langle \xi \rangle^{\sigma} - \tau \langle \eta \rangle^{\sigma} - \tau \left(1 + 1/K \right)^{-\delta/s} \langle \xi - \eta \rangle^{(1-\delta)/s} \right) .
		$$
		
		\noindent Using inequality \eqref{3.ineq.tri.1}, this implies
		\begin{eqnarray*}
			W(\tau' ; \tau) & \leq & \exp \left( - (\tau -\left( K^{\sigma} - (K-1)^{\sigma} \right)\tau') \langle \eta \rangle^{\sigma}  \right) \\
				& & \quad \times \exp \left( - \left(\tau \left(1 + 1/K \right)^{-\delta/s} - \tau' \right) \langle \xi - \eta \rangle^{(1-\delta)/s} \right)
		\end{eqnarray*}
				
		\noindent Since $K < 1$ and $\tau > \tau'$, there holds $ \tau -\left( K^{\sigma} - (K-1)^{\sigma} \right)\tau'>0$. Thus 
		$$
			W(\tau' ; \tau) \leq \exp \left( - \left(\tau \left(1 + 1/K \right)^{-\delta/s} - \tau' \right) \langle \xi - \eta \rangle^{(1-\delta)/s} \right)
		$$
		
		\noindent and if $K$ is large enough, depending only on $\tau$ and $\tau'$, there holds $\tau \left(1 + 1/K \right)^{-\delta/s} - \tau' > 0$. Thus $|W|_{L^2_{\eta}} < \infty$.
		
		$\bullet$ The case where $\frac{1}{K} |\xi - \eta| < |\eta| < K |\xi - \eta|$: here we use inequality \eqref{3.ineq.tri.2}, which implies, since $\sigma \leq (1-\delta)/s$,
		$$
			W(\tau' ; \tau) \leq \exp \left( - (\tau - \tau') \langle \eta \rangle^{\sigma} \right) \exp \left( - \left( \tau \langle \xi \rangle^{-\delta/s} - \tau' \langle \xi - \eta \rangle^{-\delta/s} \right) \langle \xi - \eta \rangle^{1/s} \right),
		$$
		
		\noindent where $c' = c'(K) \in (0,1)$. Since $|\xi| \leq (1+ K)|\xi - \eta|$ in the region under consideration, hence $\langle \xi \rangle \leq (1=K)\langle \xi - \eta \rangle$, this implies
		$$
			W \leq \exp \left( - \left( \tau (1+K)^{-\delta/s} - \tau' \right) \langle \xi - \eta \rangle^{(1-\delta)/s} \right) 
		$$
		
		\noindent thus if $K$ is large enough, depending only on $\tau$ and $\tau'$, there holds $|W|_{L^2_{\eta}} < \infty$.
	
\end{proof}


\section{A conjugation Lemma for operators}


We consider here a symbol $a$ in ${\rm S}^{m}_{\rho,0}{\rm G}^{s}_{R}$ for $\rho \in[0,1]$, $R >0$ and $s \in (1,\infty)$, with compact support $B$ of $\R^{d}_{x}$, uniformly in $\xi\in\R^{d}$. It is known (see Lemma 7.1 in \cite{bronshtein}) that there is a symbol $\widetilde{a}$ such that
\be
	{\rm op}\left( \widetilde{a} \right) = {\rm op}(a)^{(\tau)} = e^{\tau D^{\sigma}} \op(a) e^{-\tau D^{\sigma}}
\ee

\noindent and which satisfies
\be
	\label{3.def.tilde.p}
	\widetilde{a}(x,\xi) = \int_{y,\eta} e^{-i \eta \cdot y } e^{\tau \langle \xi + \eta \rangle^{\sigma} - \tau \langle \xi \rangle^{\sigma}} a(x + y,\xi) dy d\eta .
\ee

\noindent In Proposition 2.1 in \cite{bronshtein}, the symbol $\widetilde{a}$ is proved to be in $S^{m}_{1,0}$ for small $\tau$. We extend here the result for all $|\tau| < sR^{-1/s}$, with in addition an estimate of the semi-norms of the symbol. 

\begin{lemma}
	\label{3.lemma.conjugation}
	Given $a$ in ${\rm S}^{m}_{\rho,0}{\rm G}^{s}_{R}$, for any $|\tau| < sR^{-1/s}$, the symbol defined by \eqref{3.def.tilde.p} is in ${\rm S}^{m}_{1,0}$. Moreover, for any $\und{\tau} \in (|\tau|, sR^{-1/s})$, for any $\a$, $\beta$ in $\N^{d}$ there holds
	$$
		\sup_{x \in B , \,\xi \in\R^{d}} \left| \langle \xi \rangle^{-m + |\beta|} \, \d_{x}^{\a} \d_{\xi}^{\beta} \widetilde{a}(x,\xi) \right|\, \lesssim \, |B|^{1/2} C\left(\und{\tau} s^{-1}R^{1/s}\right) \sup_{\a\in\N^{d}}|a|_{\a, \beta} \, (\und{\tau} - |\tau|)^{-(2|\beta| + |\a|)/\sigma} 
	$$
	
	\noindent where constant $C$ is defined in \eqref{3.def.C.gevrey}.
	
\end{lemma}

\begin{proof}
	We compute the derivatives of the symbol $\widetilde{a}$. There holds
	\begin{align}
		& \d_{x}^{\a} \d_{\xi}^{\beta} \widetilde{a}(x,\xi) \nonumber \\
		& = \sum_{\beta_1 + \beta_2 = \beta} \binom{\beta}{\beta_1,\beta_2} \int_{y,\eta} e^{-i \eta \cdot y } \, \d_{\xi}^{\beta_1} \left( e^{\tau \langle \xi + \eta \rangle^{\sigma} - \tau \langle \xi \rangle^{\sigma}} \right) \, \d_{x}^{\a} \d_{\xi}^{\beta_2} a(x + y,\xi) \, dy d\eta \nonumber \\
		& = \sum_{\beta_1 + \beta_2 = \beta} \binom{\beta}{\beta_1,\beta_2} \int_{\eta} e^{i \eta \cdot x } \, \d_{\xi}^{\beta_1} \left( e^{\tau \langle \xi + \eta \rangle^{\sigma} - \tau \langle \xi \rangle^{\sigma}} \right) \, (i\eta)^{\a} \d_{\xi}^{\beta_2} \hat{a}(\eta,\xi) \,d\eta . \label{3.local.proof}
	\end{align}
	
	\noindent We use now the fact that $a$ is in ${\rm S}^{m}_{\rho,0}{\rm G}^{s}_{R}$ with compact support $B$ in $\R^{d}_{x}$, uniformly in $\xi$. Thanks to Proposition \ref{3.prop.x.in.xi}, we may write
	$$
		e^{ \und{\tau} \langle \eta \rangle^{\sigma}} \, \langle \xi \rangle^{-m + \rho |\beta_2|} \, \d_{\xi}^{\beta_2} \hat{a}(\eta, \xi) = F_{\xi,\beta_2}(\eta),
	$$
	\noindent where for fixed $\xi\in\R^{d}$ and $\beta_2\in\N^{d}$, $F_{\xi,\beta_2}$ is in $L^2_{\eta}$ with bound
	$$
		\left| F_{\xi,\beta_2} \right|_{L^2_{\eta}} \lesssim |B|^{1/2} C\left(\und{\tau} s^{-1}R^{1/s}\right) \sup_{\a \in \N^{d}} | a |_{\a,\beta_2} 
	$$
	
	\noindent uniformly in $\xi\in\R^d$ and $\beta_2\in\N^{d}$, and for all $\und{\tau} < sR^{-1/s}$. The semi-norms of $a$ are defined in \eqref{3.def.semi.norms}. Next, as proved in the course of Proposition 2.1 in \cite{bronshtein}, there holds 
	$$
		\left| \d_{\xi}^{\beta_1} \left( e^{\tau \langle \xi + \eta \rangle^{\sigma} - \tau \langle \xi \rangle^{\sigma}} \right) \right| \lesssim \langle \xi \rangle^{-|\beta_1|} \, \langle \eta \rangle^{2|\beta_1|} \, e^{\tau \langle \xi + \eta \rangle^{\sigma} - \tau \langle \xi \rangle^{\sigma}} .
	$$
	
	\noindent This is proved using Fa\`a di Bruno formula (see Lemma 5.1 in \cite{morisse2016lemma}) and inequality $\d_{\xi}^{\beta_1} \left( \langle \xi + \eta \rangle^{\sigma} - \langle \xi \rangle^{\sigma} \right) \lesssim \langle \xi \rangle^{-|\beta_1|} \langle \eta \rangle^{2|\beta_1|}$. The integral in \eqref{3.local.proof} satisfies thus
	\begin{align*}
		& \left| \int_{\eta} e^{i \eta \cdot x } \, \d_{\xi}^{\beta_1} \left( e^{\tau \langle \xi + \eta \rangle^{\sigma} - \tau \langle \xi \rangle^{\sigma}} \right) \, (i\eta)^{\a} \d_{\xi}^{\beta_2} \hat{a}(\eta,\xi) \,d\eta  \right| \\
		& \lesssim \int_{\eta} \, \left| \d_{\xi}^{\beta_1} \left( e^{\tau \langle \xi + \eta \rangle^{\sigma} - \tau \langle \xi \rangle^{\sigma}} \right) \right| \, |\eta|^{\a}  e^{ -\und{\tau} \langle \eta \rangle^{\sigma}} \, \langle \xi \rangle^{m-\rho|\beta_2|} \, \left| F_{\xi,\beta_2}(\eta) \right| \,d\eta \\
		& \lesssim \langle \xi \rangle^{m-|\beta_1| - \rho|\beta_2|} \, \int_{\eta} \, e^{\tau \langle \xi + \eta \rangle^{\sigma} - \tau \langle \xi \rangle^{\sigma} -\und{\tau} \langle \eta \rangle^{\sigma} } \, \langle \eta \rangle^{2|\beta_1| + |\a|}  \, \left| F_{\xi,\beta_2}(\eta) \right| \,d\eta .
	\end{align*}
	
	\noindent Next, we use inequality \eqref{3.ineq.poly.gevrey} in Lemma \ref{3.lemma.tri.ineq} to get
	$$
		\langle \eta \rangle^{2|\beta_1| + |\a|} \lesssim \, (\und{\tau} - \tau)^{ - (2|\beta_1| + |\a|)/\sigma } \,e^{ (\und{\tau} - \tau) \langle \eta \rangle^{\sigma} }
	$$
	
	\noindent hence 
	\begin{align*}
		& \int_{\eta} \, e^{\tau \langle \xi + \eta \rangle^{\sigma} - \tau \langle \xi \rangle^{\sigma} -\und{\tau} \langle \eta \rangle^{\sigma} } \, \langle \eta \rangle^{2|\beta_1| + |\a|}  \, \left| F_{\xi,\beta_2}(\eta) \right| \,d\eta \\
		& \lesssim (\und{\tau} - \tau)^{ - (2|\beta_1| + |\a|)/\sigma } \, \int_{\eta} \, e^{\tau \langle \xi + \eta \rangle^{\sigma} - \tau \langle \xi \rangle^{\sigma} - \tau \langle \eta \rangle^{\sigma} }  \, \left| F_{\xi,\beta_2}(\eta) \right| \,d\eta 
	\end{align*}
	
	\noindent and we conclude using the proof of Proposition \ref{3.prop.DmFtau}.
	
\end{proof}

We recall also the asymptotic expansion of $\widetilde{a}$, as given in Proposition 2.1 in \cite{bronshtein}.

\begin{lemma}[Asymptotic expansion of $\widetilde{a}$]
	For any $k\in\N$ there holds
	\be
		\widetilde{a}(x,\xi) = \sum_{ |\a|\leq k } \frac{i^{\a}}{\a!} \d_{x}^{\a}a(x,\xi) \left( \tau \d_{\xi} \langle \xi \rangle^{\sigma} \right)^{\a} + R
	\ee
	
	\noindent with $R$ in ${\rm S}^{\max\{m-(k+1)(1-\sigma) , m - 2 + \sigma\}}_{1,0}$.
\end{lemma}

\noindent This result is used in particular in our forthcoming papers \cite{morisse2016lemma} and \cite{morisse2016IIz}.



	\bibliographystyle{alpha}
	\bibliography{bibliographie}

\begin{thebibliography}{BMM13}

\bibitem[BMM13]{bedrossian2013landau}
Jacob Bedrossian, Nader Masmoudi, and Cl{\'e}ment Mouhot.
\newblock Landau damping: paraproducts and {G}evrey regularity.
\newblock {\em arXiv preprint arXiv:1311.2870}, 2013.

\bibitem[CNR]{bronshtein}
Ferruccio Colombini, Tatsuo Nishitani, and Jeffrey Rauch.
\newblock Weakly hyperbolic systems by symmetrization.
\newblock {\em eprint arXiv:1508.03945v2}.

\bibitem[HR01]{hua2001paradifferential}
Chen Hua and Luigi Rodino.
\newblock Paradifferential calculus in gevrey classes.
\newblock {\em Journal of Mathematics of Kyoto University}, 41(1):1--31, 2001.

\bibitem[Mor16a]{morisse2016I}
Baptiste Morisse.
\newblock On hyperbolicity and {G}evrey well-posedness. {P}art 1: the elliptic
  case.
\newblock {\em arXiv preprint arXiv:1611.07225}, 2016.

\bibitem[Mor16b]{morisse2016II}
Baptiste Morisse.
\newblock On hyperbolicity and {G}evrey well-posedness. {P}art 2: scalar or
  degenerate transitions.
\newblock {\em arXiv preprint arXiv:1611.08184}, 2016.

\bibitem[Mor17a]{morisse2016lemma}
Baptiste Morisse.
\newblock On hyperbolicity and {G}evrey well-posedness. {P}art 3: a class of
  weakly hyperbolic systems.
\newblock {\em In preparation}, 2017.

\bibitem[Mor17b]{morisse2016IIz}
Baptiste Morisse.
\newblock On hyperbolicity and {G}evrey well-posedness. {P}art 4: generic
  non-scalar transitions.
\newblock {\em In preparation}, 2017.

\bibitem[Rod93]{rodino1993linear}
Luigi Rodino.
\newblock {\em Linear partial differential operators in {G}evrey spaces}.
\newblock World Scientific, 1993.

\end{thebibliography}

\end{document}